\documentclass[12pt,A4paper]{article}

\usepackage[left=25mm,right=25mm,top=25mm,bottom=30mm]{geometry}
\usepackage{epsfig}
\usepackage{epstopdf}
\usepackage{soul}
\usepackage{xfrac}
\usepackage{amsmath}
\usepackage{enumerate}

\usepackage[percent]{overpic}

\usepackage{color}
\usepackage{amsthm,amsmath,amssymb}
\usepackage{booktabs}
\usepackage{mathpazo}
\usepackage{microtype}
\usepackage{overpic}
\usepackage{bm}
\usepackage{sectsty}
\usepackage[
	pdftitle={PDFTitle},
	pdfauthor={Hugo Parlier},
	ocgcolorlinks,
	linkcolor=linkred,
	citecolor=linkred,
	urlcolor=linkblue]
{hyperref}

\usepackage{tikz}
\usetikzlibrary{calc,decorations.pathreplacing}

\definecolor{linkred}{RGB}{157,91,246} 
\definecolor{linkblue}{RGB}{16, 78, 139}

\usepackage[hang,flushmargin]{footmisc}

\usepackage{titlesec}
	\titlespacing{\section}{0pt}{12pt}{0pt}
	\titlespacing{\subsection}{0pt}{6pt}{0pt}
	
\titlelabel{\thetitle.\quad}

\makeatother 

\theoremstyle{plain}
\newtheorem{theorem}{Theorem}[section]

\newtheorem{lemma}[theorem]{Lemma}

\theoremstyle{definition}

\newtheorem{remark}[theorem]{Remark}

\newcommand{\R}{{\mathbb R}}

\newcommand{\G}{{\mathcal G}}

\newcommand{\pitop}{N_X^{\mathrm{top}}}
\newcommand{\pitopplus}{N_X^{\mathrm{top},+}}

\newcommand{\Curr}{\mathcal C}
\newcommand{\ML}{\mathcal{ML}}
\newcommand{\Mod}{\operatorname{Mod}}
\newcommand{\Isom}{\operatorname{Isom}}
\newcommand{\Th}{\mu_{\mathrm{Thu}}}

\sectionfont{\large \bfseries}
\subsectionfont{\normalsize}

\long\def\symbolfootnote[#1]#2{\begingroup%
\def\thefootnote{\fnsymbol{footnote}}\footnote[#1]{#2}\endgroup}

\def\blfootnote{\xdef\@thefnmark{}\@footnotetext}

\usepackage{mathtools}

\usepackage[sort,nocompress]{cite}

\begin{document}

{\Large \bfseries A topological version of Huber's theorem}

{\large Ara Basmajian\symbolfootnote[1]{\small 
Supported by a Dolciani faculty research grant, a PSC-CUNY research grant, and a grant from the Simons Foundation (TSM 00013865 A.B.).},
Hugo Parlier\symbolfootnote[7]{\small Supported by ANR-SNF Grant number 200021E\_238147
(SUGAR).\\
{\em 2020 Mathematics Subject Classification:}\\Primary: 32G15, 57K20 Secondary: 30F45, 53C22.\\
{\em Key words and phrases:}\\
closed geodesics, hyperbolic surfaces, geodesic currents, topological types}
}

\vspace{0.5cm}
{\bf Abstract.}
Let $X$ be a closed hyperbolic surface. We prove that the number of topological types of primitive closed geodesics of length at most $L$ is asymptotic to 
\[
\frac{1}{|\Isom(X)|}\frac{e^L}{2L}.
\]
as $L$ grows. Thus Huber's asymptotic remains unchanged after quotienting by topological type, up to the finite symmetry factor coming from the isometry group of $X$.

\vspace{0.5cm}

\section{Introduction} \label{sec:intro}

By a result of Huber \cite{Huber}, the number of non-oriented primitive closed geodesics of a closed hyperbolic surface grows asymptotically like $e^L/(2L)$. In contrast, Mirzakhani \cite{Mirzakhani2008,MirzakhaniOrbits} studied the growth of geodesics of fixed topological type, showing that there is again asymptotic growth but this time of order $L^d$ where $d$ is the dimension of the underlying moduli space. The passage from polynomial to exponential is thus due to an increasing number of topological types. 

We make the following observation about their growth which, to the best of our knowledge, does not seem to be explicitly stated in the literature. 

\begin{theorem}\label{thm:types}
For any closed orientable hyperbolic surface $X$, as $L\to\infty$, the number of topological types of primitive closed geodesics of length less than $L$ grows asymptotically like
\[
 \frac{1}{|\Isom(X)|}\,\frac{e^L}{2 L}
 \]
where $|\Isom(X)|$ is the order of the isometry group of $X$. 
\end{theorem}

This result wasn't what we expected at first. The set of geodesics of bounded length splits into subsets of curves of the same topological type, all of which grow like a polynomial of the same degree. A first guess might be that the growth of the number of types is roughly $e^L$ divided by a polynomial. As stated, this is exactly a result of Aougab and Souto \cite{AougabSouto} who were really in pursuit of another counting problem. Their result and the validity of the theorem above rely on the behavior of so-called frequencies which determine the dominating coefficient in the polynomial growth rates of curves of fixed topological type, but the problem here is orthogonal in some sense to recent asymptotic results about frequencies where the topology is allowed to vary \cite{DGZZ,LiuRafiSoutoTrin}. 

We circumvent the problem of frequencies entirely by using the theory of geodesic currents, first introduced by Bonahon \cite{Bonahon}, used with great effect for numerous counting problems by Erlandsson and Souto (see for example \cite{ErlandssonSoutoCounting,ErlandssonSoutoBook}). In fact, the proof consists in simply putting together several existing results, namely the properness of the action of the mapping class group on the space of filling currents \cite{ErlandssonMondello}, and results of Aougab and Souto who meshed Huber's result and a theorem of Lalley \cite{Lalley} in the language of currents. 

Finally, we note that there are many other counting results about curve types. The main goal of \cite{AougabSouto} was to study the number of curve types with self-intersection at most $k$ and then let $k$ grow. On the other hand, Cahn, Fanoni and Petri \cite{CahnFanoniPetri} fixed $k$ and let the genus grow. 

\noindent {\bf Acknowledgement.}

We thank Juan Souto for a helpful conversation and encouragement. 

\section{Counting topological types}

We begin by fixing the notation and recalling the specific objects we need. We use both $\#$ and $| \cdot |$ for cardinality of a set. 

Let $\Sigma$ be a closed orientable surface of negative Euler characteristic, and $X$ a closed hyperbolic surface homeomorphic to $\Sigma$. Let $\ell_X$ denote the length function associated to $X$ which associates to the free homotopy class of closed curve the length of its corresponding closed geodesic. We denote by $\G_X(L)$ the set of primitive unoriented\footnote{Considering unoriented geodesics accounts for the ``dividing by $2$'' factor in Huber's theorem.} closed geodesics of length less than $L$ on $X$. We set $\Mod(\Sigma)$ to be the full mapping class group of $\Sigma$, and we say that two closed curves $\gamma$ and $\delta$ are of the same topological type if there exists $\phi \in \Mod(\Sigma)$ such that $\phi(\gamma)$ and $\delta$ are freely homotopic. We now denote by $\pitop(L)$ the number of topological types of curves found in $\G_X(L)$. If we fix a marking of $X$ by $\Sigma$, we can identify its full isometry group $\Isom(X)$ with its image in $\Mod(\Sigma)$.

Note that we could have restricted ourselves to defining topological type as with the orientation preserving mapping class group $\Mod^{+}(\Sigma)$, defined $\pitopplus(L)$ in the same way, and the corresponding result is explained in Remark \ref{rem:orient}. 

The proof uses several useful facts about geodesic currents, first introduced by Bonahon \cite{Bonahon}. Let $\Curr(\Sigma)$ denote the associated space of geodesic currents on $\Sigma$ and its intersection form
\[
i : \Curr(\Sigma) \times \Curr(\Sigma) \to \R_{\geq 0}
\]
which is continuous, symmetric and bilinear. A current is called filling if its support intersects every other geodesic, and we denote by $\Curr^{f}(\Sigma)$ the set of filling currents. Now for a closed hyperbolic surface $X\cong \Sigma$, let $L_X \in \Curr(\Sigma)$ be its Liouville current. The function $\ell_X$ defined above on free homotopy classes extends to a continuous function
\[
\ell_X:\Curr(\Sigma) \to \R,
\]
homogeneous under the action of $\R_+$ and such that for every $\mu \in \Curr(\Sigma)$ we have
\[
 i(L_X,\mu)=\ell_X(\mu).
\]
We consider the compact set
\[
 \Curr_X^1:=\{\mu\in\Curr(\Sigma):\ell_X(\mu)=1\}
\]
and set 
\[
 \overline{L}_X:=\frac{L_X}{\ell_X(L_X)}=\frac{L_X}{\pi^2 |\chi(X)|}
\]
the normalized Liouville current associated to $X$. The slice $\Curr_X^1$ is naturally identified with projective current space. We equip it with the induced action of $\Mod(\Sigma)$: for $\phi\in\Mod(\Sigma)$ and $\mu\in\Curr_X^1$, set
\[
 \phi\cdot\mu
 =
 \frac{\phi_*\mu}{\ell_X(\phi_*\mu)}.
\]
Thus this action is simply the usual action on projective currents, written in the length-one slice. Notice that for $h\in \Isom(X)$ this agrees with the usual linear action, since $h$ preserves $\ell_X$.

\subsection{Three ingredients}
The first ingredient is due to Aougab and Souto \cite[Lemma~2.1]{AougabSouto}.

\begin{lemma}\label{lem:growth}
For any open neighborhood $U$ of $\overline L_X$ in $\Curr_X^1$, 
\begin{equation*}
 \#\left\{\gamma\in\mathcal G_X(L):
 \frac{\gamma}{\ell_X(\gamma)}\in U\right\}
 \sim \frac{e^L}{2L}.
\end{equation*}
\end{lemma}
Thus, after normalization by length, a density-one set of closed geodesics is
arbitrarily close to the Liouville current. They obtain this by combining a theorem of Lalley's \cite{Lalley} in the language of currents with Huber's asymptotic formula \cite{Huber}. 

The second ingredient is the following, due to Erlandsson and Mondello \cite[Proposition~4.1]{ErlandssonMondello}:
\begin{lemma}\label{lem:proper}
The mapping class group acts properly discontinuously on the space of filling currents $\Curr^f(\Sigma)\subset\Curr(\Sigma)$.
\end{lemma}
The above result is also recorded explicitly as Property (9) for geodesic currents in a paper by Erlandsson and Souto \cite{ErlandssonSoutoSimpleToAll}. The same local statement can also be deduced from Burger-Iozzi-Parreau-Pozzetti \cite{BIPP} using their proper-discontinuity result on the positive-systole locus. Finally, note that the statements of the above results are stated for $\Mod^+(\Sigma)$, the index $2$ subgroup of orientation preserving self-homeomorphisms, but the properness immediately extends to the full mapping class group. Note that this also implies proper discontinuity on the subset $\Curr_X^1\subset \Curr(\Sigma)$.

We have a final needed ingredient, which bounds the number of curves left invariant by isometries. We set $G_X:=\Isom(X)$ to be the full group of isometries of $X$. 

\begin{lemma}\label{lem:isom}
The number of closed geodesics of length at most $L$ invariant by a non-trivial element of $G_X$ grows at most like $O(e^{L/2})$.
\end{lemma}

\begin{proof}
Fix $1\neq h\in G_X$ and suppose that a primitive closed geodesic $\gamma$ is preserved setwise by $h$.

If $h$ fixes $\gamma$ pointwise, then $h$ must be orientation-reversing since an orientation-preserving isometry of $X$ which fixes a geodesic pointwise is the identity. And a non-trivial orientation-reversing isometry which fixes a curve pointwise is an involution, and $\gamma$ is a subset of its fixed points. There are at most $g+1$ such curves where $g$ is the genus of $X$, so in particular at most $O(1)$ (this is referred to as Harnack's theorem). 

We can thus suppose that the restriction of $h$ to $\gamma$ is nontrivial. 

If $h$ preserves the orientation of $\gamma$, then it acts on $\gamma$ by a
nontrivial rotation of some order $k\geq 2$. Thus $\gamma$ projects to a
closed geodesic in the compact hyperbolic orbifold $ X/\langle h\rangle$ of length
\[
 \frac{\ell_X(\gamma)}{k}\leq \frac{\ell_X(\gamma)}2.
\]
If $h$ reverses the orientation of $\gamma$, then its restriction to $\gamma$ has order two because it is an orientation reversing isometry of a circle. And because $h^2$ is an orientation preserving isometry of $X$ that fixes $\gamma$ pointwise, it is equal to the identity, and so $h$ is also a global involution of $X$. 

Now if $h$ itself is orientation-preserving, then $\gamma$ projects on $X/\langle h\rangle$ to a geodesic arc of length $\ell_X(\gamma)/2$ between orbifold points. And if $h$ is orientation-reversing, then $X/\langle h\rangle$ is a surface with geodesic boundary, and $\gamma$ projects to an orthogeodesic of length $\ell_X(\gamma)/2$. 

Consequently every $h$-invariant primitive geodesic of length at most $L$ gives, up to a multiplicity bounded in terms of $h$, either a closed geodesic, a geodesic arc between two orbifold points in the fixed orbifold $X/\langle h\rangle$, or an orthogeodesic on $X/\langle h\rangle$, all of length at most $L/2$. 

The crucial observation is now the following: let $Y$ be an orbifold surface, possibly with boundary geodesics. For any $R>0$, there is a $O(e^R)$ bound on all of the following quantities:
\begin{enumerate}
\item the number of closed geodesics of $Y$ of length at most $R$,
\item the number of geodesic arcs of length at most $R$ between orbifold points,
\item the number of orthogeodesics of length at most $R$ between boundary geodesics.
\end{enumerate}
For closed geodesics this follows from the prime geodesic theorem or, more concretely, from a result of Buser's \cite[Lemma 6.6.4]{BuserBook}. The proof of the latter result is an orbit growth technique in the universal cover and the same proof shows that the number of geodesic arcs between orbifold points and the number of orthogeodesics are both bounded above by $O(e^R)$. 

Therefore, for any nontrivial $h\in \Isom(X)$, we have
\[
 \#\{\gamma\in\mathcal G_X(L):h(\gamma)=\gamma\}
 =O(e^{L/2}).
\]
Since $G_X$ is finite, summing over $1\neq h\in G_X$ proves the lemma.
\end{proof}

\subsection{The main proof}
Since $L_X$ is filling, we may choose a relatively compact neighborhood $W$ of $\overline L_X$ in $\Curr_X^1$ whose closure consists of filling currents. Shrinking $W$ if necessary, we may moreover suppose that
\[
 \frac{1}{2}\, i(\overline L_X,\overline L_X)
 <
 i(\mu,\mu)
 <
 2\,i(\overline L_X,\overline L_X)
\]
for every $\mu\in W$.

Suppose that $\mu,\nu\in W$ and that $\phi\cdot\mu=\nu$ for some $\phi\in\Mod(\Sigma)$. By definition of the induced action on $\Curr_X^1$, there is some $t>0$ such that $\phi_*\mu=t\nu$.

Since the intersection form is mapping-class-group invariant,
\[
 i(\mu,\mu)=i(\phi_*\mu,\phi_*\mu)=t^2i(\nu,\nu),
\]
and hence $t$ is bounded above and below by positive constants independent of $\phi$. It follows that both $\mu$ and $\phi_*\mu$ lie in a fixed compact subset of $\Curr^f(\Sigma)$. Proper discontinuity therefore implies that only finitely many mapping classes $\phi$ can satisfy
\[
( \phi\cdot W) \cap W\neq\varnothing.
\]
The stabilizer of $\overline L_X$ for this action is precisely $G_X=\Isom(X)$. We may therefore shrink $W$ to a $G_X$-invariant neighborhood $U$ of $\overline L_X$ such that
\begin{equation}\label{eq:localproper}
( \phi\cdot U)\cap U\neq\varnothing
 \quad\Longrightarrow\quad
 \phi\in G_X.
\end{equation}
and set
\[
 A_L:=\left\{\gamma\in\mathcal G_X(L):
 \frac{\gamma}{\ell_X(\gamma)}\in U\right\}.
\]
By Lemma \ref{lem:growth}, $|A_L|\sim e^L/(2L)$. If $\gamma,\delta\in A_L$ have the same topological type, then $\delta=\phi(\gamma)$ for some $\phi\in\Mod(\Sigma)$. By the definition of the induced action on $\Curr_X^1$,
\[
 \phi\cdot
 \frac{\gamma}{\ell_X(\gamma)}
 =
 \frac{\delta}{\ell_X(\delta)}.
\]
Hence $(\phi\cdot U)\cap U\neq\varnothing$, and \eqref{eq:localproper} implies that $\phi\in G_X$.

Furthermore, $U$ is $G_X$-invariant. Thus the topological types represented in $A_L$ are exactly the $G_X$-orbits in $A_L$. Now Lemma \ref{lem:isom} says that all but $O(e^{L/2})$ of the relevant geodesics belong to free $G_X$-orbits. Equivalently, Burnside's lemma and Lemma \ref{lem:isom} give
\[
 \#(A_L/G_X)
 =\frac{1}{|G_X|}\sum_{h\in G_X}|A_L^h|
 =\frac{|A_L|}{|G_X|}+O(e^{L/2})
 \sim \frac{1}{|G_X|}\,\frac{e^L}{2L}
\]
where $A_L^h$ denotes the subset of $A_L$ preserved by $h$.

Finally, Huber's theorem gives $|\mathcal G_X(L)|\sim e^L/(2L)$, so Lemma \ref{lem:growth} implies
\[
 |\mathcal G_X(L)\setminus A_L|
 =o\!\left(\frac{e^L}{2L}\right).
\]
Every topological type counted by $\pitop(L)$ but not represented in $A_L$ contains a geodesic in this exceptional set. Hence the number of such types is also $o(e^L/(2L))$, and  thus 
\[
 \#(A_L/G_X) \sim \pitop(L) 
\]
and Theorem~\ref{thm:types} follows.

\subsection{Final comments}
We end with a few remarks.

\begin{remark}\label{rem:orient}
If we defined topological types of curves as being up to orientation-preserving homeomorphism (or, equivalently the set of $\Mod^+(\Sigma)$ orbits), then the result becomes 
\[
 \pitopplus(L) \sim 
 \frac{1}{|\Isom^+(X)|}\,\frac{e^L}{2L}.
 \]
\end{remark}

\begin{remark}\label{rem:frequencies}
There is a useful heuristic argument which explains why the $ \approx e^L/L$ order of growth is what you might expect. This follows the work of Rafi and Souto \cite{RafiSouto}.

The dimension of the space of measured laminations on $\Sigma$ is 
\[
 d=\dim\ML(\Sigma)=6g-6
\]
where $g$ is the genus of $\Sigma$. For a filling current $\mu$, the quantity \[
 m(\mu)
 =\Th\{\lambda\in\ML(\Sigma):i(\mu,\lambda)\le1\}
\]
is the Thurston measure of the unit ball about $\mu$. Rafi-Souto identify this, up to the standard fixed normalization and finite stabilizer convention, with the type-dependent factor in the mapping-class-orbit counting asymptotic (the so-called frequency). Now since Thurston measure is homogeneous of degree $d$:
\[
 m(t\mu)=t^{-d}m(\mu).
\]
Hence, for a long filling geodesic $\gamma$, we have 
\[
m(\gamma)
 =\ell_X(\gamma)^{-d}
 m\!\left(\frac{\gamma}{\ell_X(\gamma)}\right).
\]
And, for a generic geodesic, by the work of Aougab-Souto \cite{AougabSouto}, the normalized version of the current is very close to $\overline L_X$. In particular, a generic geodesic of length about $L$ should have frequency of order $L^{-d}.$

This is exactly the scale one should expect. Mirzakhani's orbit-counting formula \cite{MirzakhaniOrbits} says that the number of mapping class orbits of $\gamma$ of length less than $L$ is asymptotic to 
\[
A_X\,c(\gamma)L^d.
\]
So a typical curve of length close to $L$ has $c(\gamma)\asymp L^{-d}$, then its topological type contributes only $\asymp 1$. The main theorem makes this rough growth estimate precise: generically the only multiplicity that survives is the finite number coming from isometries.
\end{remark}

{\it Addresses:}\\
The Graduate Center and Hunter College, CUNY, NY, New York, USA.\\
Department of Mathematics, University of Fribourg, Switzerland\\

{\it Emails:}\\
abasmajian@gc.cuny.edu\\
hugo.parlier@unifr.ch\\

\end{document}